\documentclass[letterpaper,draftcls,conference,10pt,onecolumn]{ieeeconf}
\IEEEoverridecommandlockouts
\usepackage[utf8]{inputenc}
\usepackage{amsmath}
\usepackage{amssymb}
\usepackage{bbold}
\usepackage{graphicx}
\usepackage{physics}
\usepackage{caption}
\usepackage{subcaption}
\usepackage{enumerate}
\usepackage{tikz}
\usepackage{diagbox}
\usepackage{cite}
\usetikzlibrary{calc}
\usetikzlibrary{matrix}
\usetikzlibrary{shapes}
\usetikzlibrary{positioning}
\usetikzlibrary{backgrounds}
\usetikzlibrary{arrows}
\tikzset{main node/.style={circle,fill=blue!20,draw,minimum size=1cm,inner sep=0pt},}

\newcommand{\fatone}{\mathbb{1}}

\newtheorem{corollary}{Corollary}
\newtheorem{remark}{Remark}
\newtheorem{lemma}{Lemma}
\newtheorem{theorem}{Theorem}

\newtheorem{definition}{Definition}
\newtheorem{problem}{Problem}

\usepackage{mathtools}
\mathtoolsset{showonlyrefs=true}

\title{Differentially Private Formation Control
\thanks{This work was supported by the AFOSR Center of Excellence on Assured Autonomy in Contested Environments and by NSF CAREER grant \#1943275.}}
\author{Calvin Hawkins and Matthew Hale$^*$\thanks{
$^*$The authors are with the Department of Mechanical and Aerospace Engineering,
Herbert Wertheim College of Engineering, University of Florida. Emails:
\texttt{\{calvin.hawkins,matthewhale\}@ufl.edu}.}}

\begin{document}
\maketitle
\begin{abstract}
As multi-agent systems proliferate, there is increasing demand for coordination protocols that protect agents' sensitive information while allowing them to collaborate. To help address this need, this paper presents a differentially private formation control framework. Agents' state trajectories are protected using differential privacy, which is a statistical notion of privacy that protects data by adding noise to it. We provide a private formation control implementation and analyze the impact of privacy upon the system. Specifically, we quantify tradeoffs between privacy level, system performance, and 
connectedness of the network's communication topology. These tradeoffs are used to develop guidelines for calibrating privacy in terms of control theoretic quantities, such as steady-state error, without requiring in-depth knowledge
of differential privacy. Additional guidelines are also developed for treating privacy levels and network topologies as design parameters to tune the network's performance. Simulation results illustrate these tradeoffs and show that strict privacy is inherently compatible with strong system performance. 
\end{abstract}
\section{Introduction}
Multi-agent systems, such as robotic swarms and social networks, require agents to share information to collaborate. In some cases, the information shared between agents may be sensitive. For example, self-driving cars share location data to be routed to a destination. Geo-location data and other data streams can be quite revealing 
about users and sensitive data should be protected. However, this data must still be useful for multi-agent coordination. Thus, privacy in multi-agent control must simultaneously protect agents’ sensitive data while guaranteeing that privatized data enables the network to achieve a common task.

This type of privacy has recently been achieved using differential privacy. Differential privacy stems from the computer science literature, where it was originally used to protect sensitive data when databases are queried \cite{dwork2014algorithmic,dwork2006calibrating}. Differential privacy is appealing because it is immune to post-processing and robust to side information \cite{dwork2014algorithmic}. These properties mean that privacy guarantees are not compromised by performing operations on differentially private data, and that they are not weakened by much by an adversary with additional information about data-producing 
agents~\cite{kasiviswanathan2014semantics}. 

Recently, differential privacy has been applied to dynamic systems~\cite{le2013differentially,yazdani2018differentially,hale2017cloud,le2017differentially,jones2019towards,mitraC,geirEnt, xu2020differentially, wang2016differentially}. One form of differential privacy in dynamic systems protects sensitive trajectory-valued data, and this is the notion of differential privacy used in this paper. Privacy of this form ensures that an adversary is unlikely to learn much about the state trajectory of a system by observing its outputs. In multi-agent control, this lets an agent share its outputs with other agents while protecting its state trajectory from those agents and eavesdroppers~\cite{le2013differentially,yazdani2018differentially,hale2017cloud,le2017differentially}.

In this paper, we develop a framework for private multi-agent formation control using differential privacy. Formation control is a well-studied network control problem and can be robots physically assembling into geometric shapes or non-physical agents maintaining relative state offsets. For differential privacy, agents add privacy noise to their states before sharing them with other agents. The other agents use privatized states in their update laws, and then this process repeats at every time step. Adding privacy noise makes this problem equivalent to a certain consensus protocol with measurement noises. This paper focuses on private formation control, though the methods presented can be used to design and analyze other private consensus-style protocols. The private formation control protocol can be implemented in a completely distributed manner, and, contrary to some other privacy approaches, it does not require a central coordinator.

Beyond the privacy implementation, we develop guidelines for calibrating privacy in formation control. Specifically, we bound the quality of formation, or performance of the system, in terms of agents' privacy parameters and connectedness of the network. We develop guidelines by using these bounds to trade off degraded performance for stricter privacy requirements and a less connected communication topology. This ultimately allows us to formulate privacy guidelines based on control-theoretic properties without requiring users to have an in-depth understanding of differential privacy. Guidelines are also developed 
by analyzing the sensitivity of system performance to changes in privacy parameters and communication topology to determine which
has a larger impact on system performance. 
Furthermore, we develop necessary and sufficient conditions
for when private formation control networks achieve a desired performance level.

The rest of the paper is organized as follows. Section II gives graph theory and differential privacy background. Section III states the differentially private formation control problem and Section IV solves it. Section V provides guidelines for calibrating privacy based on performance requirements for specific communication topologies. In Section VI, we analyze the sensitivity of system performance to changes in privacy and communication topology. Next, Section VII provides simulations, and Section VIII concludes the paper.
\section{Background and Preliminaries}
In this section we briefly review the required background on graph theory and differential privacy.
\subsection{Graph Theory Background}
A graph $\mathcal{G} = (V,E)$ is defined over a set of nodes $V$ and edges are contained in the set $E$. For $N$ nodes, $V$ is indexed over $\{1,...,N\}$. The edge set of $\mathcal{G}$ is a subset $E \subseteq V \times V$, where the pair $(i,j) \in E$ if nodes $i$ and $j$ share a connection and $(i,j) \notin E$ if they do not. This paper considers undirected, weighted, simple graphs. Undirectedness means that an edge $(i,j) \in E$ is not distinguished from $(j,i) \in E$. Simplicity means that $(i,i) \notin E$ for all $i \in V$. Weightedness means that the edge $(i,j) \in E$ has a weight $w_{ij} = w_{ji} >0$. Of particular interest are connected graphs.

\begin{definition}[Connected Graph]
    A graph $\mathcal{G}$ is connected if, for all $i,j \in \{1,...,N\}$, $i \neq j$, there is a sequence of edges one can traverse from node $i$ to node $j$.\hfill $\triangle$
\end{definition}

This paper uses the weighted graph Laplacian, which is defined with weighted adjacency and weighted degree matrices. The weighted adjacency matrix $A(\mathcal{G}) \in \mathbb{R}^{N \times N}$ of $\mathcal{G}$ is defined element-wise as 
\begin{equation*}
    A(\mathcal{G})_{ij} =
    \begin{cases}
    w_{ij} & (i,j) \in E \\
    0 & \text{otherwise}
    \end{cases}.
\end{equation*} Because we only consider undirected graphs, $A(\mathcal{G})$ is symmetric. The weighted degree of node $i \in V$ is defined as $d_i = \sum_{j \mid (i,j) \in E} w_{ij}.$ 
The maximum degree is~$d_{max} = \max_i d_i$. 
The degree matrix $D(\mathcal{G}) \in \mathbb{R}^{N \times N}$ is the diagonal matrix $D(\mathcal{G}) = \textnormal{diag}(d_1,...,d_N)$. The weighted Laplacian of $\mathcal{G}$ is then defined as $L(\mathcal{G}) = D(\mathcal{G}) - A(\mathcal{G})$.

Let $\lambda_k(\cdot)$ be the $k^{th}$ smallest eigenvalue of a matrix. By definition, $\lambda_1(L(\mathcal{G})) = 0$ for all graph Laplacians and $$ 0 = \lambda_1(L(\mathcal{G})) \leq \lambda_2(L(\mathcal{G})) \leq \dots \leq \lambda_N(L(\mathcal{G})).$$ The value of $\lambda_2(\mathcal{G}))$ plays a key role in this paper and is defined as follows.

\begin{definition}[Algebraic Connectivity \cite{fiedler1973algebraic}]
    The algebraic connectivity of a graph $\mathcal{G}$ is the second smallest eigenvalue of its Laplacian and $\mathcal{G}$ is connected if and only if $\lambda_2(L(\mathcal{G})) > 0$.\hfill $\triangle$
\end{definition}

Agent $i$'s neighborhood set $N(i)$ is the set of all agents agent $i$ can communicate with, defined as $N(i) = \{j \mid (i,j) \in E \}$.

\subsection{Differential Privacy Background}
This section provides a brief description of the differential privacy background needed for the remainder of the paper. More complete expositions can be found in \cite{le2013differentially,cynthia2006differential}. Overall, the goal of differential privacy is to make similar pieces of data appear approximately indistinguishable from one another. Differential privacy is appealing because its privacy guarantees are immune to post-processing \cite{cynthia2006differential}. For example, private data can be filtered without threatening its privacy guarantees \cite{le2013differentially, 9147779}. More generally, arbitrary post-hoc computations on private data do not harm differential privacy. In addition, after differential privacy is implemented, an adversary with complete knowledge of the mechanism used to implement privacy has no advantage over another adversary without mechanism knowledge \cite{dwork2014algorithmic,dwork2006calibrating}.

In this paper we use differential privacy to privatize state trajectories of mobile autonomous agents.
We consider vector-valued trajectories of the form
${Z = (Z(1),Z(2),...,Z(k),...),}$
where $Z(k) \in \mathbb{R}^d$ for all $k$. The $\ell_p$ norm of $Z$ is defined as
$\| Z \|_{\ell_p} = \left(\sum_{k=1}^{\infty} \|Z(k)\|^p_p \right)^{\frac{1}{p}}$,
where $\|\cdot\|_p$ is the ordinary $p$-norm on $\mathbb{R}^d$. Define the set

\begin{equation*}
    \ell_p^d := \{ Z \mid Z(k) \in \mathbb{R}^d, \| Z \|_{\ell_p} < \infty\}.
\end{equation*}

The set $\ell_p^d$ only contains trajectories that converge to the origin. However, we want to privatize arbitrary trajectories, including those that do not converge at all. To do so, we consider a larger set of trajectories. Let the truncation operator $P_T$ be defined as
\begin{equation*}
    P_T[y] = \begin{cases}
    y(k) & k \leq T
    \\0 & k > T
    \end{cases}.
\end{equation*}
Then we define the set
\begin{equation*}
    \tilde{\ell}_p^d = \{ Z \mid Z(k) \in \mathbb{R}^d, P_T [Z] \in \ell_p^d \text{ for all } T \in \mathbb{N}\},
\end{equation*} and we will privatize state trajectories in this set.

Consider a network of $N$ agents, where agent $i$'s state trajectory is denoted by $y_i$. The $k^{th}$ element of agent $i$'s trajectory is $y_i(k) \in \mathbb{R}^n$ for $n \in \mathbb{N}$. Agent $i$'s state trajectory belongs to $\tilde{\ell}_2^{n}$.

Differential privacy is defined with respect to an adjacency relation. We provide privacy to single agents' state trajectories (rather than collections of trajectories as in some other works), and our choice of adjacency relation is defined for single agents. In the case of dynamic systems, the adjacency relation gives a notion of how similar trajectories are and specifies which trajectories must be made approximately indistinguishable from each other.
 
\begin{definition}[Adjacency \cite{yazdani2018differentially}]
Fix an adjacency parameter $b_i > 0$ for agent $i$. $\text{Adj}_{b_i}: \tilde{\ell}_2^n \times \tilde{\ell}_2^n \xrightarrow{} \{0,1\}$ is defined as
\begin{equation*}
    \text{Adj}_{b_i}(v_i,w_i) = \begin{cases}
    1 & \|v_i - w_i\|_{\ell_2} \leq b_i
    \\ 0 & \text{otherwise.}\tag*{$\triangle$}
    \end{cases}
\end{equation*}
\end{definition}

In words, two state trajectories of agent $i$ are adjacent if and only if the $\ell_2$-norm of their difference is upper bounded by $b_i$. This means that every state trajectory within distance $b_i$ from agent $i$'s state trajectory must be made approximately indistinguishable from it to enforce differential privacy.

To calibrate differential privacy's protections, agent $i$ selects privacy parameters $\epsilon_i$ and $\delta_i$. These parameters determine the level of privacy afforded to $x_i$. Typically, $\epsilon_i \in [0.1, \ln{3}]$ and $\delta_i \leq 0.01$ for all $i$ \cite{yazdani2018differentially}. The value of $\delta_i$ can be regarded as the probability that differential privacy fails for agent $i$, while $\epsilon_i$ can be regarded as the information leakage about agent $i$.

The implementation of differential privacy in this work provides differential privacy for each agent individually. This will be accomplished by adding noise to sensitive data directly, an approach called ``input perturbation'' privacy in the literature~\cite{leny20}. The noise is added by a privacy mechanism, which is a randomized map. We now provide a formal definition of differential privacy, which states the guarantees a mechanism must provide. First, fix a probability space $(\Omega, \mathcal{F},\mathbb{P})$. We are considering outputs in $\tilde{\ell}_2^{n}$ and use a $\sigma$-algebra over $\tilde{\ell}_2^{n}$, denoted $\Sigma_2^{n}$ \cite{hajek2015random}.

\begin{definition}[Differential Privacy]
    Let $\epsilon_i > 0$ and $\delta_i \in [0,\frac{1}{2})$ be given. A mechanism $M: \tilde{\ell}_2^{n} \cross \Omega \xrightarrow{}\tilde{\ell}_2^{n}$ is $(\epsilon_i,\delta_i)$-differentially private if, for all adjacent $y_i,y_i^\prime \in \tilde{\ell}_2^{n}$, we have 
    \begin{equation*}
        \mathbb{P}[M(y_i) \in S] \leq e^{\epsilon_i}\mathbb{P}[M(y_i^\prime) \in S] + \delta_i \text{ for all } S \in \Sigma_2^{n}.\tag*{$\triangle$}
    \end{equation*}
\end{definition}

The Gaussian mechanism will be used to implement differential privacy. The Gaussian mechanism adds zero-mean i.i.d. noise drawn from a Gaussian distribution pointwise in time. Stating the required distribution uses the $Q$-function, defined as $Q(y) = \frac{1}{\sqrt{2\pi}} \int_y^{\infty} e^{-\frac{z^2}{2}}dz.$
\begin{lemma}[Gaussian Mechanism \cite{le2013differentially}]
     Let~$b_i > 0$, $\epsilon_i > 0$, and $\delta_i \in (0,\frac{1}{2})$ be given, and fix the adjacency
     relation~$\textnormal{Adj}_{b_i}$. 
     Let $y_i \in \tilde{\ell}_2^{n}$. The Gaussian mechanism for ($\epsilon_i$,$\delta_i$)-differential privacy takes the form ${\tilde{y}_i(k) = y_i(k) + w_i(k),}$
    where $w_i$ is a stochastic process with $w_i(k) \sim \mathcal{N}(0, \sigma^2_i I_{n})$ and
    $\sigma_i \geq \frac{b_i}{2\epsilon_i}(K_{\delta_i} + \sqrt{K_{\delta_i}^{2} + 2\epsilon_i})$ where $K_{\delta_i} =  Q^{-1}(\delta_i).$
    This mechanism provides ($\epsilon_i$,$\delta_i$)-differential privacy to $y_i$.\hfill $\blacksquare$
\end{lemma}

For convenience, let $\kappa(\delta_i,\epsilon_i) = \frac{1}{2\epsilon_i}(K_{\delta_i} + \sqrt{K_{\delta_i}^{2} + 2\epsilon_i}).$

\section{Problem Formulation}
In this section we state and analyze the differentially private formation control problem.
\begin{problem}
Consider a network of $N$ agents with communication topology modeled by the undirected, simple, connected, and weighted graph $\mathcal{G}$. Let $y_i(k)$ be agent $i$'s state at time $k$, $N(i)$ be agent $i$'s neighborhood set, $\gamma > 0$, and $w_{ij}$ be a positive weight on the edge~$(i, j) \in E$. We define $\Delta_{ij} \in \mathbb{R}^n$ for all $ (i,j) \in E$ as the desired relative distance between agents $i$ and $j$.
\begin{enumerate}[i.]
    \item Implement the formation control protocol
\begin{equation}
    y_i(k+1) = y_i(k) + \gamma \sum_{j \in N(i)}w_{ij}(y_j(k) - y_i(k) - \Delta_{ij}),
\end{equation}
in a differentially private manner. 
\item Analyze the relationship between network performance, privacy, and the underlying graph topology.\hfill{$\triangle$}
\end{enumerate}
\end{problem}

We will solve Problem 1 by bounding the performance of the network in terms of the privacy parameters of each agent and the algebraic connectivity of the underlying graph. This will allow us to analyze the relationship between performance, privacy, and topology.

\begin{remark} \label{rem:scalar}
We consider formation control in~$\mathbb{R}^n$, which is equivalent to running~$n$ scalar-valued formation controllers.
Therefore, we analyze the scalar case. This simplifies the forthcoming analysis while also giving a granular
error analysis that allows for controlling error in each dimension independently. 
The~$n$-dimensional controller is simply an~$n$-fold replication of the scalar-valued controller.
If an overall error bound
for all dimensions is necessary, one need only multiply the forthcoming one-dimensional error bounds by the number
of dimensions,~$n$. 
\end{remark}

Before solving Problem 1, we give the necessary definitions for formation control. First, we define agent- and network-level dynamics. Then, we detail how each agent will enforce differential privacy. Lastly, we explain how differentially private communications affect the performance of a formation control protocol and how to quantify quality of a formation.
\subsection{Multi-agent Formation control}
The goal of formation control is for agents in a network to assemble into some geometric shape or set of relative states. 
Multi-agent formation control is a well researched problem and there are several mathematical formulations one can use to achieve similar results \cite{jadbabaie2015scaling,krick2009stabilisation,ren2007information,ren2007consensus,fax2004information,olfati2007consensus,mesbahi2010graph}. We will define relative distances between agents that communicate and the control objective is for all agents to maintain the relative distances to each of their neighbors. This approach is similar to that of \cite{ren2007information} and the translationally invariant formations presented in \cite{mesbahi2010graph}.

For the formation to be feasible, $\Delta_{ij} = -\Delta_{ji}$ for all $(i,j) \in E$. The network control objective is driving ${\lim_{k \xrightarrow{} \infty} (y_j(k) - y_i(k)) = \Delta_{ij}}$ for all $(i,j) \in E.$ It is important to note that there is an infinite set of points that can be in formation; the formation can be centered around any point in $\mathbb{R}^n$ and meet the control requirement, i.e., we allow formations to be translationally invariant \cite{mesbahi2010graph}.

Now we define the agents' update law. Let $\{p_1,...,p_N\}$ be any collection of points in formation such that $p_j-p_i=\Delta_{ij}$ for all $(i,j) \in E$ and let $p = (p_1^T,\dots,p_N^T)^T \in \mathbb{R}^{nN}$ be the network-level formation specification. We consider the
formation control protocol
\begin{equation}
    y_i(k+1) = y_i(k) + \gamma \sum_{j \in N(i)}w_{ij}(y_j(k) - y_i(k)-\Delta_{ij}).\label{nodelevel_no_noise}
\end{equation}
As noted in Remark~\ref{rem:scalar}, we analyze convergence
of Equation~\eqref{nodelevel_no_noise} at the component level. Thus, while~$y_i \in \mathbb{R}^n$, we select
an arbitrary~$l \in \{1, \ldots, n\}$ and provide analysis
for
\begin{equation}
x(k) = \left(\begin{array}{c} y_{1,l}(k) \\ \vdots \\ y_{N,l}(k) \end{array}\right) \in \mathbb{R}^N,
\end{equation}
i.e., each agents~$l^{th}$ component, which proceeds identically for each~$l \in \{1, \ldots, n\}$. Below, we also use the vector of~$l^{th}$ components of~$p$, denoted
\begin{equation}
q = \left(\begin{array}{c} p_{1,l} \\ \vdots \\ p_{N,l} \end{array}\right).  
\end{equation}
Let $\bar{x}(k) = x(k) - q$. Then we analyze
\begin{equation}
    \bar{x}(k+1) = (I - \gamma L(\mathcal{G}))\bar{x}(k). \label{networklevel_nonoise}
\end{equation}
Letting $P = I - \gamma L(\mathcal{G})$, we may write $\bar{x}(k+1) = P\bar{x}(k)$. In this form, we have the following convergence result.
\begin{lemma}[\cite{olfati2007consensus}, Theorem~2] 
If $\mathcal{G}$ is connected, $P$ is doubly stochastic, and $\gamma \in (0,\frac{1}{d_{max}})$, then the protocol in Equation
\eqref{networklevel_nonoise} reaches consensus asymptotically 
and~${\bar{x}(k) \xrightarrow{} \fatone^T\frac{1}{n}\bar{x}(0) \fatone.}$ \hfill $\blacksquare$
\end{lemma}

Because the protocol in Equation \eqref{networklevel_nonoise} reaches consensus over $\bar{x}$, it solves the translationally invariant formation control problem \cite{mesbahi2010graph}. 
Using~$\delta_{ij}$ to denote the state offset between agents~$j$ and~$i$ in the appropriate dimension,
the node-level protocol in Equation \eqref{nodelevel_no_noise} can be rewritten for a single component as
\begin{equation} \label{eq:smalldelta}
    x_i(k+1) = x_i(k) + \gamma \sum_{j \in N(i)}w_{ij}(x_j(k) - x_i(k) - \delta_{ij}),
\end{equation}
which we use below. 
\subsection{Private Communications}
When agent $j$ transmits $\bar{x}_j(k)$ to the agents in $N(j)$, it is potentially exposing its state trajectory, $x_j$, to them and adversaries or eavesdroppers. Agent $j$ therefore sends a differentially private version of $\bar{x}_j(k)$ to its neighborhood.

Agent $j$ starts by selecting privacy parameters $\epsilon_j > 0$, $\delta_j \in (0, \frac{1}{2})$, and adjacency relation $\text{Adj}_{b_j}$ with $b_j > 0$. Agent~$j$ then privatizes its state trajectory $x_j$ with the Gaussian mechanism. Let $\tilde{x}_j$ denote the differentially private version of $x_j$, where, pointwise in time, $\tilde{x}_j(k) = x_j(k) + v_j(k),$ with $v_j(k) \sim \mathcal{N}(0,\sigma_j^2)$ and $\sigma_j \geq \kappa(\delta_j,\epsilon_j)b_j$. Thus agent $j$ keeps the trajectory $x_j$ differentially private. Agent $j$ then shares $\tilde{\bar{x}}_j(k) = \tilde{x}_j(k) - q_j$, which is also differentially private because subtracting $q_j$ is merely post-processing \cite{cynthia2006differential}.

\subsection{Private Formation Control}
When each agent is sharing differentially private information, the node-level formation control protocol becomes
\begin{equation}
\bar{x}_i(k+1) = \bar{x}_i(k) + \gamma \sum_{j \in N(i)}w_{ij}(\tilde{\bar{x}}_j(k) - \bar{x}_i(k))\label{DP_node_level},
\end{equation} where agent $i$ uses $\bar{x}_i$ rather than $\tilde{\bar{x}}_i$ because it always has access to its own unprivatized state. 
The stochastic nature of this protocol implies that agents no longer exactly reach a formation, and, in
particular, the states will never exactly converge to a steady-state value.

To analyze performance, let $$\beta(k) := \frac{1}{N}\fatone^T x(k)\fatone + q  - \frac{1}{N}\fatone^T q\fatone,$$ which is the state vector the protocol in Equation \eqref{networklevel_nonoise}  would converge to with initial state $x(k)$ and without privacy. 
Also let ${e(k) = x(k) - \beta(k),}$
which is the distance of the current state to the state the protocol would converge to without differential privacy. To quantity the effects of privacy on the network as a whole, let $e_{\textnormal{agg}}(k) := \frac{1}{n}\sum_{i=1}^n E[e_i^2(k)]$
be the aggregate error of the network, and let
\begin{equation}
    e_{ss} := \limsup_{k\xrightarrow{} \infty} e_{\textnormal{agg}}(k)\label{steadystateerror}
\end{equation}
be the steady-state error of the network.

Problem 1 requires us to quantify the relationship between privacy, encoded by $(\epsilon_i, \delta_i)$; performance, encoded by $e_{ss}$; and topology, encoded by $\lambda_2$. These quantitative tradeoffs are the subject of the next section.
\section{Differentially Private Formation Control}
In this section we solve Problem 1. First, we show how the private formation control protocol can be modeled as a Markov chain. Then, we solve Problem 1 by deriving performance bounds that are functions of the underlying graph topology and each agent's privacy parameters.
\subsection{Formation Control as a Markov chain}
Problem 1 takes the form of a consensus protocol with Gaussian i.i.d. noise perturbing each agent's state, which has been previously studied in \cite{jadbabaie2015scaling}. We begin by expanding $\tilde{\bar{x}}_j(k)$ in Equation \eqref{DP_node_level}, which yields
\begin{equation}
    \bar{x}_i(k+1) = \bar{x}_i(k) + \gamma \sum_{j \in N(i)}w_{ij}(\bar{x}_j(k) + v_j(k) - \bar{x}_i(k)).\label{Node_level_noise_ugly}
\end{equation}
For the purposes of analysis, we will consider equivalent network-level dynamics given as follows.
\begin{lemma}
     Let agents use the communication graph $\mathcal{G}$ with weighted Laplacian $L(\mathcal{G})$. Then Equation \eqref{Node_level_noise_ugly} can be represented at the network level as $\bar{x}(k+1) = P \bar{x}(k) + z(k),$
    where $P = I - \gamma L(\mathcal{G})$ and $z(k) \sim \mathcal{N}(0,Z)$ where $Z =\textnormal{diag}(s_1^2,...,s_N^2)$, with $s_i^2 =\gamma^2 \sum_{j \in N(i)}w_{ij}^2 \sigma_j^2$.
\end{lemma}

\begin{proof}
	Using Equation~\eqref{eq:smalldelta}, 
    Equation~\eqref{Node_level_noise_ugly} can be expanded as
    \begin{equation*}
    \begin{split}
        x_i(k+1) = x_i(k) + \gamma \sum_{j \in N(i)}w_{ij}(x_{j}(k) - x_i(k) - \delta_{ij}) \\+ \gamma \sum_{j \in N(i)}w_{ij}v_j(k).
    \end{split}
    \end{equation*}
    The last term encodes privacy noise. Without this term we have formation control without noise, which can be represented at the network level as $\bar{x}(k+1) = P \bar{x}(k).$
    Next, let 
    ${z_i(k) = \gamma \sum_{j \in N(i)}w_{ij}v_j(k)}$. Using the fact that ${E[v_j(k)] = 0}$
    and 
    ${\textnormal{Var}[\sum_{j=1}^N v_j(k)] = \sum_{j=1}^N \textnormal{Var}[v_j(k)],}$
    we have ${E[z_i(k)]=0}$
    and ${\textnormal{Var}[z_i(k)] =\gamma^2 \sum_{j \in N(i)}w_{ij}^2 \sigma_j^2.}$ The lemma follows from setting
    ${s_i^2 =\gamma^2 \sum_{j \in N(i)}w_{ij}^2 \sigma_j^2}$. 
\end{proof}

For analysis, we use the network-level update law
\begin{equation}
        \bar{x}(k+1) = P \bar{x}(k) + z(k).\label{network_level}
\end{equation}
The main result of this paper uses the fact that a stochastic matrix $P$ can serve as the transition matrix of a Markov chain and the properties of the Markov chain can be used to analyze the network dynamics associated with $P$. Before stating our main results we first define the conditions under which a network modeled by an undirected, weighted, connected graph can be modeled as a Markov chain and establish the properties of this Markov chain.

\begin{lemma}
For an undirected, weighted, simple, connected graph $\mathcal{G}$, let $\gamma > 0$ be given. If for all $i$ the graph weights are designed such that $\gamma \sum_{j \in N(i)} w_{ij} < 1,$
then the matrix $P = I - \gamma L(\mathcal{G})$ is doubly stochastic.
\end{lemma}

\begin{proof}
    The weighted graph Laplacian has row $i$
    \begin{equation*}
    L_{\textnormal{row~$i$}}(\mathcal{G}) =
        \begin{bmatrix}
        -w_{i1} & \hdots & \sum_{j \in N(i)}w_{ij} & \hdots & -w_{in}\\
        \end{bmatrix}.
    \end{equation*}
    Then row $i$ of $P$ is row $i$ of $I - \gamma L(\mathcal{G})$, which is
    \begin{equation*}
    P_{\textnormal{row~$i$}} =
        \begin{bmatrix}
        \gamma w_{i1}  & \hdots & 1 -\gamma \sum_{j \in N(i)}w_{ij} & \hdots & \gamma w_{in}\\
        \end{bmatrix}.
    \end{equation*}
    If agents $i$ and $j$ are not neighbors, then $w_{ij} = 0$, which implies that $\sum_{j \in N(i)} w_{ij}=\sum_{\substack{j=1 \\ j\neq i}}^N w_{ij}$. Then
    \begin{equation*}
        \sum_{j=1}^N{P_{ij}} = 1 - \gamma \sum_{j \in N(i)} w_{ij} + \gamma \sum_{\substack{j=1 \\ j\neq i}}^N w_{ij} = 1 \text{ for all } i.
    \end{equation*}
    Because $w_{ij} \geq 0$ and $ \gamma \sum_{j \in N(i)} w_{ij} < 1$, every entry of $P$ is greater than or equal to 0. Because the sum of every row of $P$ is $1$ and every entry is greater than or equal to 0, $P$ is row stochastic. The same procedure can be used to show that $P$ is column stochastic because $w_{ij} = w_{ji}$ for all $(i,j) \in E$. Therefore $P$ is doubly stochastic.
\end{proof}

Graph Laplacian properties can be used to make stronger statements.  In particular, the Laplacian of an undirected graph is always symmetric, and therefore $P$ is symmetric. Let the stationary distribution of the Markov chain be $\pi$, which satisfies $\pi^T P = \pi^T.$
With the symmetry of $P$ we have the following explicit form for $\pi.$
\begin{lemma}
     If $P$ is symmetric, then its stationary distribution is $\pi = \frac{1}{N} \fatone.$
\end{lemma}
\begin{proof}
    See \cite[Chapter 4]{pinsky2010introduction}.
\end{proof}
Furthermore, we can make a stronger statement about $P$.
\begin{lemma}
     Let $\gamma \in \left(0,\frac{1}{d_{max}}\right)$. If $\mathcal{G}$ is connected, simple, and finite, then $P = I - \gamma L(\mathcal{G})$ is irreducible, aperiodic, positive recurrent, and reversible. 
\end{lemma}
\begin{proof}
    The fact that $\mathcal{G}$ is connected and finite implies $P = I - \gamma L(\mathcal{G})$ is irreducible \cite[Chaper 1]{levin2017markov}. $\mathcal{G}$ being connected and simple along with $\gamma \in \left(0,\frac{1}{d_{max}}\right)$ and Lemma 4 imply that $P$ has self loops and thus is aperiodic \cite[Chapter 1]{levin2017markov}. The existence of $\pi$ implies positive recurrence \cite[Theorem 21.12]{levin2017markov}. Lastly, the above properties along with the symmetry of $P$ allow us to use Kolmogorov's Criteria to deduce that $P$ is reversible \cite[Section 1.5]{kelly2011reversibility}.
\end{proof}

Lemmas 4-6 allow us to develop bounds on the steady state error in Equation \eqref{steadystateerror} using \cite{jadbabaie2015scaling}. That work details several specific cases of consensus protocols with noise vectors of the form in Equation \eqref{network_level}. However, for this paper we are only interested in the results when $P$ is symmetric and the noise is i.i.d, which take the following form.

\begin{lemma}[From~\cite{jadbabaie2015scaling}]
If $P$ is irreducible, aperiodic, and reversible, and if the noises at the nodes are uncorrelated, such that the off diagonal elements of $Z$ are $0$ and $z(k) \sim \mathcal{N}(0,Z)$, then $e_{ss}$ is bounded via
     \begin{equation*}
        \left(\min_{1 \leq i \leq N} s_i^2 \pi_i\right)K(P^2) \leq e_{ss} \leq \left(\max_{1 \leq i \leq N} s_i^2 \pi_i\right)K(P^2),
    \end{equation*} where $K(P^2)$ is the Kemeny constant of the Markov chain with transition matrix $P^2$.\hfill$\blacksquare$
\end{lemma}

In this work, given that $P = I - \gamma L(\mathcal{G})$, we wish to relate $K(P^2)$ to the agents' graph topology encoded in $L(\mathcal{G})$. We do so with the following bound.

\begin{lemma}
     Let $P$ be the transition matrix of a finite, irreducible, and reversible Markov chain. Let $\lambda_2(P)$ be the second largest eigenvalue of $P$. Then the Kemeny constant $K(P)$ is bounded via $\frac{N-1}{2} < K(P) \leq \frac{N-1}{1-\lambda_2(P)}$
    and $$\frac{N-1}{2} < K(P^2) \leq \frac{N-1}{1-(1 - \gamma\lambda_2(L(\mathcal{G})))^2}.$$
\end{lemma}
\emph{Proof:}
The bounds on $K(P)$ are derived in \cite{levene2002kemeny}. These results also imply $\frac{N-1}{2} < K(P^2) \leq \frac{N-1}{1-\lambda_2(P^2)}.$
Then, because $\lambda_2(P^2) = \lambda_2(P)^2 = [1 - \gamma\lambda_2(L(\mathcal{G}))]^2$, we have
\begin{equation*}
    \frac{N-1}{2} < K(P^2) \leq \frac{N-1}{1-(1 - \gamma\lambda_2(L(\mathcal{G})))^2}. \tag*{$\blacksquare$}
\end{equation*}

\subsection{Solving Problem 1}
Now we state the first of our main results: a bound on performance in terms of agents' level of privacy and underlying graph topology.
\begin{theorem}
Consider the network-level private formation control protocol $\bar{x}(k+1) =(I - \gamma L(\mathcal{G}))\bar{x}(k) + z(k).$
If $\gamma \sum_{j \in N(i)} w_{ij} < 1$, $\gamma \in \left(0,\frac{1}{d_{max}}\right),$ $\mathcal{G}$ is connected and undirected, and $\sigma_i \geq \kappa(\delta_i,\epsilon_i)b_i$ for all $i$, then $e_{ss}$ is upper-bounded by
\begin{equation*}
    e_{ss} \leq \frac{\gamma(N-1)^2 \max_i  \kappa(\delta_i,\epsilon_i)^2b_i^2}{N\lambda_2(L(\mathcal{G}))(2 - \gamma\lambda_2(L(\mathcal{G})))}.
\end{equation*}
\end{theorem}
\begin{proof}
With Lemma 7, $e_{ss} \leq \left(\max_i s_i^2 \pi_i\right)K(P^2).$
Then using Lemma 8 to upper bound $K(P^2)$ gives
\begin{equation*}
    e_{ss} \leq \frac{\left(\max_i s_i^2 \pi_i\right)(N-1)}{\gamma\lambda_2(L(\mathcal{G}))(2 - \gamma\lambda_2(L(\mathcal{G})))}.
\end{equation*}
Recalling from Lemma 3 that $s_i^2 = \gamma^2 \sum_{j \in N(i)}w_{ij}^2 \sigma_j^2$, we have
\begin{equation*}
    e_{ss} \leq \frac{\left(\max_i \pi_i \gamma^2 \sum_{j \in N(i)}w_{ij}^2 \sigma_j^2\right)(N-1)}{\gamma\lambda_2(L(\mathcal{G}))(2 - \gamma\lambda_2(L(\mathcal{G})))}.
\end{equation*}
Then, because $\pi_i = \frac{1}{N}$ for all $i$ and $w_{ij} \in (0, 1)$,
\begin{equation*}
    \max_i \pi_i \gamma^2 \sum_{j \in N(i)}w_{ij}^2 \sigma_j^2 \leq \frac{\gamma^2}{N} \max_i\left( |N(i)| \max_{j \in N(i)} \sigma_j^2\right).
\end{equation*}
For $N$ agents $|N(i)| \leq N-1$, ${\sigma_i^2 = \kappa(\delta_i,\epsilon_i)^2 b_i^2}$, and $\max_i[\max_{j \in N(i)} \sigma_j^2] = \max_i \sigma_i^2$, which gives
\begin{equation*}
    \frac{\gamma^2}{N} \max_i[ |N(i)| \max_{j \in N(i)} \sigma_j^2] \leq \frac{\gamma^2(N-1)}{N} \max_i \kappa(\delta_i,\epsilon_i)^2b_i^2.
\end{equation*}
Plugging this result back into the upper bound gives
\begin{equation*}
    e_{ss} \leq \frac{\gamma(N-1)^2 \max_i  \kappa(\delta_i,\epsilon_i)^2b_i^2}{N\lambda_2(L(\mathcal{G}))(2 - \gamma\lambda_2(L(\mathcal{G})))}.
\end{equation*}
\end{proof}

We can simplify Theorem 1 when each agent has the same privacy parameters. 
Next, and from this point on, we consider the case where ${\sigma = \kappa(\delta,\epsilon) b}$ so that each agent adds the minimum amount of noise needed to 
attain~$(\epsilon,\delta)$-differential privacy. 
 
\begin{corollary}[Homogeneous Privacy Parameters]
Let each agent in the network have the privacy parameters $\epsilon$ and $\delta$ and the adjacency parameter $b$. Then
\begin{equation*}
    e_{ss} \leq \frac{\gamma  \kappa(\delta,\epsilon)^2 b^2(N-1)^2}{N\lambda_2(L(\mathcal{G}))(2 - \gamma\lambda_2(L(\mathcal{G})))}.
\end{equation*}
\end{corollary}

The rest of the paper focuses on the homogeneous case presented in Corollary 1, though all forthcoming results are easily adapted to the heterogeneous case by considering minima and maxima over all agents where appropriate.
\section{Network Design Guidelines}
In this section we give guidelines for designing a differentially private formation control network. 
Consider $N$ agents with homogeneous privacy parameters $\epsilon$ and $\delta$. The goal is to design the network so that~$e_{ss}$ does not exceed a given limit $e_R$. 
The question of interest is: Given a specific communication topology, how much privacy is each agent allowed to have for $e_{ss} \leq e_R$?
As noted in Remark~\ref{rem:scalar}, we do this for each dimension of formation control individually. 
 A smaller value of $\epsilon$ corresponds to being more private. Therefore an upper bound on $e_{ss}$, which is the measure of system performance, implies a lower bound on $\epsilon$, each agent's privacy parameter.

We derive an impossibility result and sufficient conditions in terms of $\epsilon$ for $e_{ss} \leq e_R$ for specific networks.  We consider connected graphs $\mathcal{G}$ with uniform weights, where $w_{ij} = w$ for all $(i,j)\in E$. By construction, the graphs we consider in this paper are weight-balanced, which implies that for any weights, the protocol in Equation \eqref{networklevel_nonoise} will converge to the unweighted average as seen in Lemma 2. Throughout this section we fix $\delta$ to be some small number and let $\epsilon$ vary to tune the level of privacy, which is common in differential privacy implementations \cite{hsu2014differential}.

In general, a more connected topology, i.e., one with larger~$\lambda_2(\mathcal{G}),$ can accommodate stronger privacy and still achieve the desired level of performance. 
This principle is used to determine the conditions under which it is impossible to satisfy~$ e_{ss} \leq e_R$, shown next.

\begin{theorem}[Impossibility Result]
Given a network of $N$ agents with specified $\epsilon$, $\delta$, $b$, and $e_R$, compute~$\lambda_2(\mathcal{G})$. 
Then~$e_{ss} \leq e_R$ cannot be assured if 
\begin{equation*}
    \epsilon < \frac{2 b z_1}{ N e_R \lambda_2(\mathcal{G})} \left(b+\frac{e_R K_{\delta} \lambda_2(\mathcal{G}) N }{\sqrt{e_R z_1\lambda_2(\mathcal{G})N}}\right),
\end{equation*}
where $z_1 = \frac{\gamma (N-1)^2}{2-\gamma\lambda_2(\mathcal{G})}.$
\end{theorem}
\emph{Proof:} 
Being unable to achieve $e_{ss} \leq e_R$ is equivalent to requiring $$\frac{\gamma  \kappa(\delta,\epsilon)^2 b^2(N-1)^2}{N\lambda_2(L(\mathcal{G}))(2 - \gamma\lambda_2(L(\mathcal{G})))} > e_R.$$ This holds if and only if $$\kappa(\delta,\epsilon)^2 b^2 > \frac{N\lambda_2(L(\mathcal{G}))(2 - \gamma\lambda_2(L(\mathcal{G})))e_R}{\gamma (N-1)^2}.$$
Then, using ${\kappa(\delta,\epsilon)^2 = \frac{1}{4\epsilon^2}(K_{\delta} + \sqrt{K_{\delta}^{2} + 2\epsilon})^2}$, we solve for $\epsilon$, giving
\begin{equation*}
    \epsilon < \frac{2 b z_1}{ N e_R \lambda_2(\mathcal{G})} \left(b+\frac{e_R K_{\delta} \lambda_2(\mathcal{G}) N }{\sqrt{e_R z_1\lambda_2(\mathcal{G})N}}\right). \tag*{$\blacksquare$}
\end{equation*}

We now derive necessary and sufficient conditions for
assuring~$e_{ss} \leq e_R$ for
 common graphs: the complete graph, line graph, cycle graph, and star graph. These conditions can easily be checked a priori and give a network designer a simple means of determining whether a specific network will meet performance requirements.
Proofs of Corollaries 3-5 are similar to that of Corollary~2 and are omitted. 

\begin{corollary}[Complete graph]
    The complete graph has algebraic connectivity $\lambda_2(L(\mathcal{G})) = wN$ \cite{de2007old}.
    Consider a network of $N$ agents with specified $\epsilon$, $\delta$, $\gamma$, $b$, $w$, and $e_R$, and communication topology modeled by the complete graph. The network can be shown to satisfy~$e_{ss} \leq e_R$ if and only if
    \begin{equation*}
    \epsilon \geq \frac{2 b \gamma  (N-1)^2}{ N^2 e_R w (2 - \gamma w N)} \left(b+\frac{e_R K_{\delta} w N \sqrt{2 - \gamma w N)}}{(N-1) \sqrt{e_R\gamma w }}\right).
\end{equation*}
\end{corollary}
\emph{Proof:} 
The complete graph satisfies the design requirements if $e_{ss} \leq e_R$. This is assured if~$\frac{\gamma  \kappa(\delta,\epsilon)^2 b^2(N-1)^2}{N\lambda_2(L(\mathcal{G}))(2 - \gamma\lambda_2(L(\mathcal{G})))} \leq e_R$, which holds if and only if $$\kappa(\delta,\epsilon)^2 b^2 \leq \frac{N\lambda_2(L(\mathcal{G}))(2 - \gamma\lambda_2(L(\mathcal{G})))e_R}{\gamma (N-1)^2}.$$
For the complete graph with all weights equal to $w$, we have $\lambda_2(L(\mathcal{G})) = wN$ and thus we require
$$\kappa(\delta,\epsilon)^2 b^2 \leq \frac{wN^2(2 - \gamma w N)e_R}{\gamma (N-1)^2}.$$
Then using ${\kappa(\delta,\epsilon)^2 = \frac{1}{4\epsilon^2}(K_{\delta} + \sqrt{K_{\delta}^{2} + 2\epsilon})^2}$ we solve the inequality for $\epsilon$ to find
\begin{equation*}
    \epsilon \geq \frac{2 b \gamma  (N-1)^2}{ N^2 e_R w (2 - \gamma w N)} \left(b+\frac{e_R K_{\delta} w N \sqrt{2 - \gamma w N)}}{(N-1) \sqrt{e_R\gamma w }}\right). \tag*{$\blacksquare$}
\end{equation*} 

\begin{corollary}[Cycle Graph]
    The cycle graph has algebraic connectivity $\lambda_2(L(\mathcal{G})) = 2 w\left(1 - \cos(\frac{2\pi}{N})\right)$ \cite{de2007old}.
    Consider a network of $N$ agents with specified $\epsilon$, $\delta$, $\gamma$, $b$, $w$, and $e_R$, and communication topology modeled by the cycle graph. The network is assured
    to satisfy~$e_{ss} \leq e_R$ if and only if $$\epsilon \geq\frac{b z_2}{ N e_R 2 w\left(1 - \cos(\frac{2\pi}{N})\right)} +\frac{K_{\delta}}{\sqrt{z_2 e_R w \left(1 - \cos(\frac{2\pi}{N})\right)N}},$$ where
    $z_2 = \frac{(N-1)^2 \gamma}{1 - \gamma  w\left(1 - \cos(\frac{2\pi}{N})\right)}.$
    \end{corollary}

\begin{corollary}[Line Graph]
     The line graph has algebraic connectivity $\lambda_2(L(\mathcal{G})) = 2w\left(1 - \cos(\frac{\pi}{N})\right)$ \cite{de2007old}.
    Consider a network of $N$ agents with specified $\epsilon$, $\delta$, $\gamma$, $b$, $w$, and communication topology modeled by the line graph. The network 
	is assured to satisfy~$e_{ss} \leq e_R$ if and only if    
    $$\epsilon \geq\frac{b z_3}{ N e_R 2 w\left(1 - \cos(\frac{\pi}{N})\right)} +\frac{K_{\delta}}{\sqrt{z_3e_R w \left(1 - \cos(\frac{\pi}{N})\right)N}},$$ where
    $z_3 = \frac{(N-1)^2 \gamma}{1 - \gamma  w\left(1 - \cos(\frac{\pi}{N})\right)}.$
\end{corollary}

\begin{corollary}[Star Graph]
    The star graph has algebraic connectivity $\lambda_2(L(\mathcal{G})) = w$ \cite{de2007old}. Consider a network of $N$ agents with specified $\epsilon$, $\delta$, $\gamma$, $b$, $w$, and $e_R$, and communication topology modeled by the star graph. The network 
    is assured to satisfy~$e_{ss} \leq e_R$ if and only if $${\epsilon \geq\frac{2 b \gamma  (N-1)^2}{ N e_R w(2 - \gamma w)} \left(b+\frac{e_R K_{\delta} w N \sqrt{2 - \gamma w)}}{(N-1) \sqrt{e_R\gamma w N}}\right).}$$
\end{corollary}

\begin{remark}
Fix $\delta = 0.01$, $b = 5$,$w = 1$, $\gamma = 10^{-4}$, and $e_R  = 100$. The lower bounds on $\epsilon$ found in Corollaries 2-5 were calculated numerically for networks with a varying number of agents, the results of which are in Table~\ref{tab:graphs}. 

\begin{table}
\centering
\begin{tabular}{| c | c | c | c | c |}
    \hline 
     \diagbox[width = .32\linewidth]{Graph}{N}& $10$ & $100$ & $1,000$ & $10,000$ \\ \hline
    Complete & $0.0074$ & $0.0081$ & $0.0084$ & $0.0116$ \\ \hline
    Cycle & $0.0380$ & $1.4514$ & $199.35$ & $159591$ \\ \hline
    Line & $0.7533$ & $3.2127$ & $714.70$ & $635752$ \\ \hline
    Star & $0.0235$ & $0.0820$ & $0.2661$ & $0.8849$ \\ \hline
  \end{tabular}
  \caption{
Comparison of the lower bounds on $\epsilon$ for various communication topologies and numbers of agents. This table illustrates that more-connected graphs accommodate privacy better when the network size grows, because they allow $\epsilon$ to be smaller, which gives stronger privacy protections.  
  }
  \label{tab:graphs}
  \end{table}
  \smallskip
\end{remark}

\section{Sensitivity Results}
Theorem 1 and Corollaries 2-5 show that performance of a network is a function of the network topology, each agent's privacy parameters, adjacency relationship, step size, and the number of agents. Some of these parameters are global, in that the parameter depends on the entire network, and some are local, in that the parameter can change at the agent level. For example, the network communication topology is a global parameter while each agent's privacy parameter, $\epsilon$, is a local parameter. 

Consider the following example: Given a network that is not performing as desired, one option is to change the network's topology and allow more agents to communicate, while another option is loosening the agents' privacy requirements. Depending on design constraints, it may be more effective to allow more agents to communicate or to relax privacy requirements. It is useful to understand when changing $\epsilon$ is more effective than changing the network topology and vice versa. In this section we therefore analyze how sensitive network performance is to local changes in privacy and global changes in topology.

We start by letting $e_{ss}$ equal the upper bound found earlier, $e_{ss} = \frac{\gamma (K_{\delta} + \sqrt{K_{\delta}^{2} + 2\epsilon})^2b^2(N-1)^2}{4\epsilon^2N\lambda_2(L(\mathcal{G}))(2 - \gamma\lambda_2(L(\mathcal{G})))}$.
Define~$\lambda(\delta,\epsilon) = K_{\delta} + \sqrt{K_{\delta}^2 + 2\epsilon}$. 
Now we take the partial derivatives with respect to $\epsilon$ and $\lambda_2(L(\mathcal{G}))$:
\begin{align} \label{partial_eps}
    \pdv{e_{ss}}{\epsilon} &= \frac{\gamma(N-1)^2b^2}{4N\lambda_2(2 - \gamma\lambda_2)}
    \left(\frac{2 \lambda(\delta, \epsilon)}{\epsilon^2 \sqrt{(K_{\delta}^2 + 2\epsilon)}} 
          - \frac{2\lambda^2(\delta,\epsilon)}{\epsilon^3}\right) \\
    \pdv{e_{ss}}{\lambda_2} &= \frac{\gamma \lambda^2(\delta,\epsilon) (N-1)^2b^2}{4\epsilon^2 N\lambda_2(2 - \gamma\lambda_2)}\left(\frac{\gamma}{2 - \gamma \lambda_2} - \frac{1}{\lambda_2}\right).
\end{align}

These partial derivatives give an understanding of how much $e_{ss}$ changes with a change in either $\epsilon$ or the network topology. As mentioned previously, we would like to determine when changing one is more effective than the other, which leads to the following result.

\begin{theorem}
Let $\alpha = \epsilon^2+\frac{3 \epsilon K_{\delta}^2}{2} + \frac{1}{\gamma ^2}+\frac{K_{\delta}^4}{2},
$
\begin{equation*}
    \eta_1 = \frac{2 \epsilon \gamma +\gamma  K_{\delta}^2+2}{2 \gamma }+\frac{1}{2} \sqrt{2 \epsilon K_{\delta}^2+K_{\delta}^4},
\end{equation*}
and $\eta_2 = \frac{2 \epsilon \gamma +\gamma  K_{\delta}^2+2}{2 \gamma }-\frac{1}{2} \sqrt{2 \epsilon K_{\delta}^2+K_{\delta}^4}.$

Then~$e_{ss}$ is more sensitive to $\lambda_2$ than $\epsilon$ when 
\begin{equation*}
   \lambda_2 > \eta_1-\sqrt{\frac{K_{\delta}^2(4 \epsilon^2 + 4 \epsilon
   K_{\delta}^2 + K_{\delta}^4) }{2\sqrt{2 \epsilon K_{\delta}^2+K_{\delta}^4}}+ \alpha},
\end{equation*}
or when
\begin{equation*}
    \lambda_2< \eta_2-\sqrt{-\frac{K_{\delta}^2(4 \epsilon^2 + 4 \epsilon
   K_{\delta}^2 + K_{\delta}^4) }{2\sqrt{2 \epsilon K_{\delta}^2+K_{\delta}^4}}+ \alpha}.
\end{equation*}
\end{theorem}
\begin{proof}
$\pdv{e_{ss}}{\lambda_2}$ and $\pdv{e_{ss}}{\epsilon}$ are both negative, so $e_{ss}$ is more sensitive to $\lambda_2$ when $\pdv{e_{ss}}{\lambda_2} < \pdv{e_{ss}}{\epsilon}$. This occurs when $\left(\frac{\epsilon \gamma}{A}-\gamma \right) \lambda_2^2 + \left(2 + \epsilon \gamma -\frac{2 \epsilon}{A}\right) \lambda_2 - \epsilon < 0$
where ${A = (K_{\delta} + \sqrt{K_{\delta}^2 + 2\epsilon})\sqrt{(K_{\delta}^2 + 2\epsilon)}}$. This inequality is satisfied by the bounds in the Theorem statement.
\end{proof}
\begin{remark}
    These results can be formulated in such a way that there is some cost associated with changing $\lambda_2(L(\mathcal{G}))$ and a cost associated with changing $\epsilon$. Making an optimal change to achieve performance criteria will largely depend on application. This will be explored
    in a future publication. 
\end{remark}

The results presented in Theorem 3 can be instantiated for specific graphs. For example, consider the following.

\begin{corollary}
    Let $\delta = 0.00135$, such that $K_{\delta} = 3$, and let $\epsilon = 0.01$. Let $\gamma = \frac{1}{10}$. The network's performance is more sensitive to the network topology than $\epsilon$ when $\lambda_2 > 5.55134.$
    
    To illustrate these results, consider the following. The star graph over $N= 10$ nodes has~$\lambda_2 =1$, which implies that the network's performance is more sensitive to changes in the privacy parameter~$\epsilon$. The complete graph over $N = 10$ nodes has~$\lambda_2 = 10$, which implies the network's performance is more sensitive to changes in the network topology.
\end{corollary}

\section{Simulation Results}
In this section, we present private formation control simulation results and illustrate the results in Theorem 1.
\subsection{Simulation of Differentially Private Formation Control}
Consider a network of $N=5$ agents. Agents $i$'s state at time $k$ is $x_i(k) \in \mathbb{R}^2$, and every agent's state trajectory is in $\tilde{\ell}_2^2$. The agents' communication topology is modeled by the star graph over $5$ nodes with weights $w_{ij} = 1$ for all $(i,j)\in E$. The network's algebraic connectivity is $\lambda_2 = 1$. The formation specification is
$$
p =\begin{bmatrix}
0 & -20 & 20 & 20 & -20\\
0 & 20 & 20 & -20 & - 20
\end{bmatrix}^T,
$$
where row $i$ denotes agent $i$'s desired location in the formation. Thus $p$ specifies a formation where agents $2$-$5$ will form a square with agent $1$ at the center.

We consider the homogeneous case where each agent has identical privacy parameters, $(\epsilon_i,\delta_i) =(\ln 3, 0.00135)$ for all $i$ and every agent also has an identical adjacency parameter $b_i = 2$ for all $i$. Let $\gamma = \frac{1}{5}$. Let $e_{11}$ denote the error of the first element of agent $1$'s state. The protocol in Equation \eqref{DP_node_level} was run for $100$ time steps

\begin{figure}
    \centering
    \includegraphics[width = .35 \textwidth]{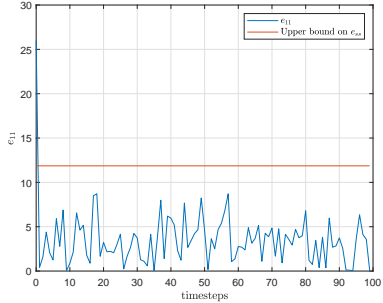}
    \caption{Agent 1's error in the first element of its state and the upper bound on $e_{ss}$. The upper bound is on the steady state value of $e_{ss}$, however it holds point-wise in time for $e_{11}$ and components of other agents' states.}
    \label{error}
\end{figure}

\begin{figure}
    \centering
    \includegraphics[width = .35\textwidth]{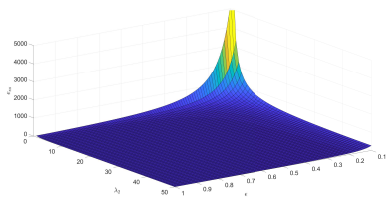}
    \caption{A plot of~$e_{ss}$ for different levels of privacy and connectedness. We fix $\delta = 0.01$ and $N = 50$ and let $\epsilon$ and $\lambda_2$ vary. It can be seen that the largest error occurs when agents are not well-connected, which gives small $\lambda_2$, and when agents keep information very private, which gives small $\epsilon$.}
    \label{surfaceplot}
\end{figure}

Figure \ref{error} shows $e_{11}$ at every time step as well as the upper bound found in Theorem 1, where we see that $e_{11}$ never converges to $0$ due to the stochastic nature of the protocol, but remains in some neighborhood of $0$. The bound on $e_{ss}$ presented in Theorem 1 is on the expected steady state value of square aggregate error, though we see that this bound also holds point-wise in time for $e_{11}$ in this simulation. These results were typical throughout numerous simulation runs.
\subsection{Theorem 1 Visualization}
In Figure \ref{surfaceplot} we include a visualization of the Results of Theorem 1. The figure considers $N=50$ agents and considers the homogeneous case where each agent has the same privacy parameters. We fix $\delta=0.01$, $\gamma=0.02$, and $b=5$. The purpose of this figure is to show the effects of changing the graph topology and privacy parameter $\epsilon$ on the upper bound on $e_{ss}$. As the graph topology becomes more connected, $\lambda_2$ increases monotonically and in Figure \ref{surfaceplot} we see that if we fix an epsilon, $e_{ss}$ decreases as the topology becomes more connected. A smaller value of $\epsilon$ corresponds to a stricter privacy requirement. Figure \ref{surfaceplot} shows that as our privacy requirements become more strict, the upper bound on $e_{ss}$ increases. 
Figure~\ref{surfaceplot} depicts~$\epsilon \in [0.1,1]$ and $\lambda_2 \in [0,50]$.

\section{Conclusions}
In this paper, we have studied the problem of differentially private formation control. This work enables agents to assemble formations while only sharing differentially private output data with a bounded steady state error. We developed guidelines for calibrating privacy under different control-theoretic requirements. The tunable parameters in this work are the privacy parameters and the topology itself, balancing the corresponding trade offs is a subject of future work.

\bibliographystyle{IEEEtran}
\bibliography{mybib}

\begin{thebibliography}{10}
\providecommand{\url}[1]{#1}
\csname url@samestyle\endcsname
\providecommand{\newblock}{\relax}
\providecommand{\bibinfo}[2]{#2}
\providecommand{\BIBentrySTDinterwordspacing}{\spaceskip=0pt\relax}
\providecommand{\BIBentryALTinterwordstretchfactor}{4}
\providecommand{\BIBentryALTinterwordspacing}{\spaceskip=\fontdimen2\font plus
\BIBentryALTinterwordstretchfactor\fontdimen3\font minus
  \fontdimen4\font\relax}
\providecommand{\BIBforeignlanguage}[2]{{%
\expandafter\ifx\csname l@#1\endcsname\relax
\typeout{** WARNING: IEEEtran.bst: No hyphenation pattern has been}%
\typeout{** loaded for the language `#1'. Using the pattern for}%
\typeout{** the default language instead.}%
\else
\language=\csname l@#1\endcsname
\fi
#2}}
\providecommand{\BIBdecl}{\relax}
\BIBdecl

\bibitem{dwork2014algorithmic}
C.~Dwork and A.~Roth, ``The algorithmic foundations of differential privacy,''
  \emph{Foundations and Trends{\textregistered} in Theoretical Computer
  Science}, vol.~9, no. 3--4, pp. 211--407, 2014.

\bibitem{dwork2006calibrating}
C.~Dwork, F.~McSherry, K.~Nissim, and A.~Smith, ``Calibrating noise to
  sensitivity in private data analysis,'' in \emph{Theory of cryptography
  conference}.\hskip 1em plus 0.5em minus 0.4em\relax Springer, 2006, pp.
  265--284.

\bibitem{kasiviswanathan2014semantics}
S.~P. Kasiviswanathan and A.~Smith, ``On the'semantics' of differential
  privacy: A bayesian formulation,'' \emph{Journal of Privacy and
  Confidentiality}, vol.~6, no.~1, 2014.

\bibitem{le2013differentially}
J.~Le~Ny and G.~J. Pappas, ``Differentially private filtering,'' \emph{IEEE
  Transactions on Automatic Control}, vol.~59, no.~2, pp. 341--354, 2013.

\bibitem{yazdani2018differentially}
K.~Yazdani, A.~Jones, K.~Leahy, and M.~Hale, ``Differentially private lq
  control,'' \emph{arXiv preprint arXiv:1807.05082}, 2018.

\bibitem{hale2017cloud}
M.~T. Hale and M.~Egerstedt, ``Cloud-enabled differentially private multiagent
  optimization with constraints,'' \emph{IEEE Transactions on Control of
  Network Systems}, vol.~5, no.~4, pp. 1693--1706, 2017.

\bibitem{le2017differentially}
J.~Le~Ny and M.~Mohammady, ``Differentially private mimo filtering for event
  streams,'' \emph{IEEE Transactions on Automatic Control}, vol.~63, no.~1, pp.
  145--157, 2017.

\bibitem{jones2019towards}
A.~Jones, K.~Leahy, and M.~Hale, ``Towards differential privacy for symbolic
  systems,'' in \emph{2019 American Control Conference (ACC)}.\hskip 1em plus
  0.5em minus 0.4em\relax IEEE, 2019, pp. 372--377.

\bibitem{mitraC}
\BIBentryALTinterwordspacing
Z.~Huang, S.~Mitra, and G.~Dullerud, ``Differentially private iterative
  synchronous consensus,'' in \emph{Proceedings of the 2012 ACM Workshop on
  Privacy in the Electronic Society}, ser. WPES ’12.\hskip 1em plus 0.5em
  minus 0.4em\relax New York, NY, USA: Association for Computing Machinery,
  2012, p. 81–90. [Online]. Available:
  \url{https://doi.org/10.1145/2381966.2381978}
\BIBentrySTDinterwordspacing

\bibitem{geirEnt}
Y.~{Wang}, Z.~{Huang}, S.~{Mitra}, and G.~E. {Dullerud}, ``Differential privacy
  in linear distributed control systems: Entropy minimizing mechanisms and
  performance tradeoffs,'' \emph{IEEE Transactions on Control of Network
  Systems}, vol.~4, no.~1, pp. 118--130, 2017.

\bibitem{xu2020differentially}
Z.~Xu, K.~Yazdani, M.~T. Hale, and U.~Topcu, ``Differentially private
  controller synthesis with metric temporal logic specifications,'' in
  \emph{2020 American Control Conference (ACC)}.\hskip 1em plus 0.5em minus
  0.4em\relax IEEE, 2020, pp. 4745--4750.

\bibitem{wang2016differentially}
Y.~Wang, M.~Hale, M.~Egerstedt, and G.~E. Dullerud, ``Differentially private
  objective functions in distributed cloud-based optimization,'' in \emph{2016
  IEEE 55th Conference on Decision and Control (CDC)}.\hskip 1em plus 0.5em
  minus 0.4em\relax IEEE, 2016, pp. 3688--3694.

\bibitem{fiedler1973algebraic}
M.~Fiedler, ``Algebraic connectivity of graphs,'' \emph{Czechoslovak
  mathematical journal}, vol.~23, no.~2, pp. 298--305, 1973.

\bibitem{cynthia2006differential}
C.~Dwork, ``Differential privacy,'' \emph{Automata, languages and programming},
  pp. 1--12, 2006.

\bibitem{9147779}
K.~{Yazdani} and M.~{Hale}, ``Error bounds and guidelines for privacy
  calibration in differentially private kalman filtering,'' in \emph{2020
  American Control Conference (ACC)}, 2020, pp. 4423--4428.

\bibitem{leny20}
J.~Le~Ny, \emph{Differential Privacy for Dynamic Data}.\hskip 1em plus 0.5em
  minus 0.4em\relax Springer, 2020.

\bibitem{hajek2015random}
B.~Hajek, \emph{Random processes for engineers}.\hskip 1em plus 0.5em minus
  0.4em\relax Cambridge university press, 2015.

\bibitem{jadbabaie2015scaling}
A.~Jadbabaie and A.~Olshevsky, ``Scaling laws for consensus protocols subject
  to noise,'' 2015.

\bibitem{krick2009stabilisation}
L.~Krick, M.~E. Broucke, and B.~A. Francis, ``Stabilisation of infinitesimally
  rigid formations of multi-robot networks,'' \emph{International Journal of
  control}, vol.~82, no.~3, pp. 423--439, 2009.

\bibitem{ren2007information}
W.~Ren, R.~W. Beard, and E.~M. Atkins, ``Information consensus in multivehicle
  cooperative control,'' \emph{IEEE Control systems magazine}, vol.~27, no.~2,
  pp. 71--82, 2007.

\bibitem{ren2007consensus}
W.~Ren, ``Consensus strategies for cooperative control of vehicle formations,''
  \emph{IET Control Theory \& Applications}, vol.~1, no.~2, pp. 505--512, 2007.

\bibitem{fax2004information}
J.~A. Fax and R.~M. Murray, ``Information flow and cooperative control of
  vehicle formations,'' \emph{IEEE transactions on automatic control}, vol.~49,
  no.~9, pp. 1465--1476, 2004.

\bibitem{olfati2007consensus}
R.~Olfati-Saber, J.~A. Fax, and R.~M. Murray, ``Consensus and cooperation in
  networked multi-agent systems,'' \emph{Proceedings of the IEEE}, vol.~95,
  no.~1, pp. 215--233, 2007.

\bibitem{mesbahi2010graph}
M.~Mesbahi and M.~Egerstedt, \emph{Graph theoretic methods in multiagent
  networks}.\hskip 1em plus 0.5em minus 0.4em\relax Princeton University Press,
  2010.

\bibitem{pinsky2010introduction}
M.~Pinsky and S.~Karlin, \emph{An introduction to stochastic modeling}.\hskip
  1em plus 0.5em minus 0.4em\relax Academic press, 2010.

\bibitem{levin2017markov}
D.~A. Levin and Y.~Peres, \emph{Markov chains and mixing times}.\hskip 1em plus
  0.5em minus 0.4em\relax American Mathematical Soc., 2017, vol. 107.

\bibitem{kelly2011reversibility}
F.~P. Kelly, \emph{Reversibility and stochastic networks}.\hskip 1em plus 0.5em
  minus 0.4em\relax Cambridge University Press, 2011.

\bibitem{levene2002kemeny}
M.~Levene and G.~Loizou, ``Kemeny's constant and the random surfer,'' \emph{The
  American mathematical monthly}, vol. 109, no.~8, pp. 741--745, 2002.

\bibitem{hsu2014differential}
J.~Hsu, M.~Gaboardi, A.~Haeberlen, S.~Khanna, A.~Narayan, B.~C. Pierce, and
  A.~Roth, ``Differential privacy: An economic method for choosing epsilon,''
  in \emph{2014 IEEE 27th Computer Security Foundations Symposium}.\hskip 1em
  plus 0.5em minus 0.4em\relax IEEE, 2014, pp. 398--410.

\bibitem{de2007old}
N.~M.~M. De~Abreu, ``Old and new results on algebraic connectivity of graphs,''
  \emph{Linear algebra and its applications}, vol. 423, no.~1, pp. 53--73,
  2007.

\end{thebibliography}

\end{document}